\documentclass[11pt]{article}
\usepackage{amsfonts,latexsym}
\usepackage{amssymb}
\usepackage{amsmath,amsthm}
\usepackage{enumerate}
\usepackage{framed}
\usepackage{verbatim}

\newtheorem{thm}{Theorem}[section]
\allowdisplaybreaks
\usepackage{mathtools}
\usepackage{bbm}
\mathtoolsset{showonlyrefs}

\theoremstyle{definition}
\newtheorem{defn}[thm]{Definition}
\newtheorem{lem}[thm]{Lemma}

\newtheorem{cor}[thm]{Corollary}
\newtheorem{tm}[thm]{Theorem}
\newtheorem{rem}[thm]{Remark}

\theoremstyle{remark}
\newtheorem{claim}[thm]{{\bf Claim}}

\newtheorem*{notn}{Notation}

\renewcommand{\P}{\mathbb{P}}

\newcommand{\Bin}{\text{Bin}}

\usepackage{bm}

\usepackage[pdftex]{graphicx}

\usepackage{pstricks,pst-node,pst-text,pst-3d}
\usepackage{pstricks}
\newrgbcolor{lilablassblau}{.7 0 .5}

\def\t{\tau}
\newcommand{\brac}[1]{\left(#1\right)}
\newcommand{\bfrac}[2]{\brac{\frac{#1}{#2}}}
\newcommand{\set}[1]{\left\{#1\right\}}
\def\cP{{\cal P}}
\def\cH{{\cal H}}

\def\G{\Gamma}
\def\e{\epsilon}

\def\cL{{\cal L}}
\def\n{\nu}
\def\l{\lambda}
\def\th{\theta}
\def\s{\sigma}
\def\tS{{\bf T_S}}
\def\dS{{\bf D_S}}

\def\tO{\tilde{O}}

\newcommand{\beq}[2]{\begin{equation}\label{#1}#2\end{equation}}
\newcommand{\proofstart}{{\noindent \bf Proof\hspace{2em}}}
\newcommand{\proofend}{\hspace*{\fill}\mbox{$\Box$}\\ \medskip\\ \medskip}
\def\bU{{\bf U}}
\begin{document}
\date{}

\title{Packing Hamilton Cycles Online}

\author{Joseph Briggs\thanks{Department of Mathematical Sciences, Carnegie Mellon University, Pittsburgh PA. {\tt jbriggs@andrew.cmu.edu}.}, Alan Frieze\thanks{Department of Mathematical Sciences, Carnegie Mellon University, Pittsburgh PA. {\tt alan@random.math.cmu.edu}. Research supported in part by NSF Grants DMS1362785, CCF1522984 and a grant(333329) from the Simons Foundation.}, Michael Krivelevich\thanks{School of Mathematical Sciences, Raymond and Beverly Sackler Faculty of Exact Sciences, Tel Aviv University, Tel Aviv 6997801, Israel. {\tt krivelev@post.tau.ac.il}. Research supported in part by a USA-Israel
    BSF Grant and by a grant from Israel Science Foundation.
    },\\ Po-Shen Loh\thanks{Department of Mathematical Sciences, Carnegie Mellon University,
    Pittsburgh, PA 15213. {\tt ploh@cmu.edu}. Research supported in part by
    NSF grant DMS-1455125.}, Benny Sudakov\thanks{Department of Mathematics, ETH, 8092 Zurich, Switzerland. {\tt benjamin.sudakov@math.ethz.ch}. Research supported in part by SNSF grant 200021-149111.}}

\maketitle
\begin{abstract}
It is known that w.h.p. the hitting time $\t_{2\s}$ for the random graph process to have minimum degree $2\s$ coincides with the hitting time for $\s$ edge disjoint Hamilton cycles, \cite{BF}, \cite{KS}, \cite{KO}. In this paper we prove an online version of this property. We show that, for a fixed integer $\s\ge 2$, if random edges of $K_n$ are presented one by one then w.h.p. it is possible to color the edges online with $\s$ colors so that at time $\t_{2\s}$, each color class is Hamiltonian.
\end{abstract}
AMS Classification Codes: 05C80, 05C15.
\section{Introduction}
The celebrated \emph{random graph process}, introduced by Erd\H{o}s~and
R\'enyi~\cite{ErdRen} in the 1960's, begins with an empty graph on
$n$ vertices, and at every step $t=1,\ldots,\binom{n}{2}$ adds to the
current graph a single new
edge chosen uniformly at random out of all missing edges.
Taking a snapshot of the random graph process after $m$ steps produces the
distribution $G_{n,m}$.
An equivalent ``static" way of defining $G_{n,m}$ would be: choose
$m$ edges uniformly at random out of all $\binom{n}{2}$ possible
ones. One advantage in studying the random graph process, rather
than the static model, is that it allows for a higher resolution
analysis of the appearance of monotone graph properties (a graph
property is monotone if it is closed under edge addition).

A \emph{Hamilton cycle} of a graph is a simple cycle that passes through
every vertex of the graph, and a graph containing a Hamilton cycle
is called \emph{Hamiltonian}. Hamiltonicity is one of the most fundamental
notions in graph theory, and has been intensively studied in various
contexts, including random graphs. The earlier results on
Hamiltonicity of random graphs were obtained by P\'osa~\cite{Po},
and Korshunov \cite{Korshunov}. Improving on these results, Bollob\'as~\cite{Bollobas}, and
Koml\'os~and Szemer\'edi~\cite{KomSze} proved that if $m' =
\frac{1}{2}n \log n + \frac{1}{2} n\log \log n + \omega n$, then $G_{n,m'}$ is Hamiltonian w.h.p.
Here $\omega$ is any function of $n$ tending to infinity together with $n$. One obvious necessary condition for the graph to be Hamiltonian is
for the minimum degree to be at least 2, and the above result indicates
that the events of being Hamiltonian and of having all degrees at least two
are indeed bundled together closely. Bollob\'as~\cite{Bollobas}, and
independently, Ajtai, Koml\'os, and Szemer\'edi \cite{AKS}, further strengthened this by proving that w.h.p. the random graph process becomes Hamiltonian when the last vertex of degree one disappears.
A more general property $\cH_\s$ of having $\s$ edge disjoint Hamilton
cycles was studied by  Bollob\'as and Frieze \cite{BF}. They showed that if
$\s=O(1)$ then w.h.p. 
the random graph process satisfies $\cH_\s$ when the minimum degree becomes
$2\s$. It took quite a while, but this result was extended to the more
difficult case of growing $\s$ in the $G_{n,m}$ context by Knox, K\"uhn and Osthus \cite{KO} and Krivelevich and Samotij \cite{KS}.

Recently, quite a lot of attention and research effort has been
devoted to controlled random graph processes. In processes of this
type, an input graph or a graph process is usually generated fully
randomly, but then an algorithm has access to this random input
and can manipulate it in some well defined way (say, by dropping
some of the input edges, or by coloring them), aiming to achieve
some preset goal. There is usually the so-called {\em online}
version where the algorithm must decide on its course of action
based only on the history of the process so far and without assuming
any familiarity with future random edges. For example, in the so-called
\emph{Achlioptas process} the random edges arrive in batches of size $k$. 
An online algorithm chooses one of them and puts
it into the graph. By doing this one can attempt to accelerate or to delay the appearance 
of some property. Hamiltonicity in Achlioptas processes was studied in \cite{KLS}.
Another online result on Hamiltonicity was proved in \cite{LSV}. There, it
was shown that one can orient the edges of the
random graph process so that w.h.p. the resulting graph has a directed
Hamilton cycle exactly at the time when the underlying graph has minimum degree two.

Here we consider a Ramsey-type version of controlled random
processes. In this version, the incoming random edge, when it is exposed, is irrevocably colored by an
algorithm in one of $r$ colors, for a fixed $r\ge 2$.  The goal of
the algorithm is to achieve or to maintain a certain monotone graph property in
all of the colors. For example, in \cite{BFKLS} the authors considered the problem of creating a linear size (so-called {\em giant}) component in every color.

The above mentioned result of Bollob\'as and Frieze \cite{BF} gives rise to the following natural question. 
Can one typically construct $\s$ edge disjoint Hamilton cycles in an online fashion by the time the minimum degree becomes $2\s$? 
We answer this question affirmatively in the case $\s=O(1)$. 

\begin{tm}\label{th1}
For a fixed integer $\s\ge 2$, let $\t_{2\s}$ denote the hitting time for
the random graph process $G_i,i=1,2,\ldots$ to have minimum degree $2\s$.
Then w.h.p. we can color the edges of $G_i,i=1,2,..$ online with $\s$ colors so that $G_{\t_{2\s}}$ contains $\s$ Hamilton cycles $C_1,C_2,\ldots,C_\s$, where the edges of cycle $C_j$ all have color $j$.
\end{tm}

\section{Description of the coloring procedure}
We describe our coloring procedure in terms of $q=2\s$ colors we aim to color the edges so that each vertex has degree at least one in each color. Think of colors 1 and $1+\s$ being light red and dark red, say, and then that each vertex is incident with at least two red edges. This may appear cumbersome, but it does make some of the description of the analysis a little easier.

In the broadest terms, we construct two sets of edges $E^+$ and $E^*$. Let $\G_c^*$ be the subgraph of $G_{\t_{2\s}}$ induced by the edges of color $c$ in $E^*$. We ensure that w.h.p. this has minimum degree at least one for all $c$. We then show that w.h.p. after merging colors $c$ and $c+\s$ for $c\in[\s]$ the subgraph $\G_c^{**}=\G_c^*\cup \G_{c+\s}^*$ has sufficient expansion properties so that standard arguments using P\'osa rotations can be applied. For every color $c$, the edges of $E^*_c$ are used to help create a good expander, and produce a backbone for rotations. And the edges in $E^+_c$ are used to close cycles in this argument.
\begin{notn}
 ``At time $t$'' is taken to mean ``when $t$ edges have been revealed".
\end{notn}

\begin{notn}
Let $N^{(t)}(v)$ denote the set of neighbors of $v$ in $G_t$ and let $d^{(t)}_v=|N^{(t)}(v)|$.

For color $c \in [q]$, write $d_c=d_{c,t},N_c=N_{c,t}$ for the degrees and neighborhoods of vertices and sets in $\G_c$.
\end{notn}
\begin{defn}\label{Fudef}
Let $Full$ denote the set of vertices with degree at least $\frac{\epsilon \log n}{1000q}$ in every color at time 
$$
t_\epsilon := \epsilon n \log n\,,
$$
where $\e$ is some sufficiently small constant depending only on the constant $q$. The actual value of $\e$ needed will depend on certain estimates below being valid, in particular equation \eqref{qdef}.
A vertex is $Full$ if is lies in $Full$. Similarly, let $Full'\subseteq Full$ denote the set of vertices with degree at least $\frac{\epsilon \log n}{1000q}$ in every color at time $\frac12\epsilon n \log n$.
\end{defn}
 This definition only makes sense if $t_\e$ is an integer. Here and below we use the following convention. If we give an expression for an integer quantity that is not clearly an integer, then rounding the expression up or down will give a value that can be used to satisfy all requirements.
\subsection{Coloring Algorithm COL}
We now describe our algorithm for coloring edges as we see them. At any time $t$, vertex $v$ has a list $C_v^{(t)}:=\{c \in [q]:d_c^{(t)}(v)=0 \}$ of colors currently not present among edges incident to $v$; ``the colors that $v$ needs''. A vertex is {\em needy} at time $t$ if $C_v^{(t)}\neq\emptyset$. If the next edge to color contains a needy vertex then we try to reduce the need of this vertex. Otherwise, we make choices to guarantee expansion in $E^*$, needed to generate many endpoints in the rotation phase, and to provide edges for $E^+$, which are used to close cycles, if needed. 
\medskip

\parindent 0in
{\bf FOR} $t=1,2,\ldots,\t_{q}$ {\bf DO}\\
{\bf BEGIN}\\
\hspace{.25in}\parbox{5in}{
\begin{enumerate}[Step 1]
\item Let $e_t=uv$.
\item If $C_v^{(t)} \cup C_u^{(t)}= \emptyset$, $t > t_{\epsilon}$, and precisely one of $\{u,v\}$ (WLOG $u$) is $Full$, then give $uv$ the color $c$ that minimises $d_c(v)$ (breaking ties arbitrarily).  Add $uv$ to $E_c^*$.
\item If $C_v^{(t)} \cup C_u^{(t)}= \emptyset$, $t > t_{\epsilon}$ and \emph{both} $u,v \in Full$, give $uv$ a color $c$ uniformly at random from $[q]$. Then add this edge to $E^+_c$ or $E^*_c$, each with probability $1/2$.
\item If $C_v^{(t)} \cup C_u^{(t)}=\emptyset$ but $t \leq t_{\epsilon}$ or \emph{both} $u,v \notin Full $, then color $uv$ with color $c$ chosen uniformly at random from $[q]$. Add $uv$ to $E_c^*$.
\item Otherwise, color $uv$ with color $c$ chosen uniformly at random from $C_u^{(t)} \cup C_v^{(t)}$. Add $uv$ to $E_c^*$.
\end{enumerate}
}\\
{\bf END}

Let 
$$E^*=\bigcup_{c\in[q]}E_c^*\text{ and }E^+=\bigcup_{c\in[q]}E_c^+.$$
\section{Structural properties}
Let
$$p=\frac{\log n + (q-1) \log \log n - \omega}{n}\text{ and }m=\binom{n}{2}p$$
where 
$$\omega=\omega(n)\to\infty,\omega=o(\log\log n).$$

We will use the following well-known properties relating $G_{n,p}$ and $G_{n,m}$, see for example \cite{FK}, Chapter 1. Let $\cP$ be a graph property. It is monotone increasing if adding an edge preserves it, and is monotone decreasing if deleting an edge preserves it. We have:
\begin{align}
&\P(G_{n,m}\in\cP)\leq 10m^{1/2}\P(G_{n,p}\in\cP).\label{Prop1}\\
&\P(G_{n,m}\in\cP)\leq 3\P(G_{n,p}\in\cP), \text{ if $\cP$ is monotone}.\label{Prop2}
\end{align}

A vertex $v\in [n]$ is {\em small} if its degree $d(v)$ in $G_{n,m}$ satisfies $d(v)<\frac{\log n}{100q}$. It is {\em large} otherwise. The set of small vertices is denoted by $SMALL$ and the set of large vertices is denoted by $LARGE$.
\begin{defn}
A subgraph $H$ of $G_{n,m}$ with a subset $S(H) \subset V(H)$ is called a \emph{small structure} if
$$|E(H)|+|S(H)|-|V(H)| \geq 1.$$
\end{defn}
We say that $G_{n,m}$ {\em contains} $H$ if there is an injective homomorphism $\phi:H\hookrightarrow G_{n,m}$ such that $\phi(S(H))\subseteq SMALL$.
The important examples of $H$ include:
\begin{itemize}
\item A single edge between 2 \emph{small} vertices.
\item A path of length at most five between two \emph{small} vertices.
\item A copy of $C_3$ or $C_4$ with at least one \emph{small} vertex.
\item Two distinct triangles sharing at least one vertex.
\end{itemize}
\begin{lem}\label{lem1}
For any fixed small structure $H$ of constant size,
$$\P(G_{n,m}\text{ contains }H)=o(n^{-1/5}).$$
\end{lem}
\begin{proof}
We will prove that
\beq{q1}{
\P(G_{n,p}\text{ contains }H)=o(n^{-3/4}).
}
This along with \eqref{Prop1} implies the lemma.

Let $h=|V(H)|,f=|E(H)|,s=|S(H)|$ so that $f+s\geq h+1$. Then:
\begin{align*}
& \P\left(G_{n,p}\text{ contains }H\right) \leq \binom{n}{h}h! p^f \left( \sum_{i=0}^{\frac{\log n}{100q}}
\binom{n-h}{i}
 p^i(1-p)^{n-h-i}\right)^s\\
& \lesssim n^h
\left( \frac{\log n}{n} \right)^f\left( \sum_{i=0}^{\frac{\log n}{100q}} \left(
 \frac{(e+o(1)) \log n}{i} \right)^i
 e^{-\log n -(q-1)\log \log n +\omega +o(1)}
 \right)^s \\
& \leq n^h\left(\frac{\log n}{n}\right)^f\left(\frac{(300q)^{\frac{\log n}{100q}}}  {n (\log n)^{q-1-o(1)} } \right)^s\\
&=o(n^{h-f-s+1/4})=o(n^{-3/4}).
\end{align*}
(We used the notation $A\lesssim B$ in place of $A\leq (1+o(1))B$.) In the calculation above, in the first line we placed the vertices of $H$ and decided about the identity of $s$ vertices falling into $SMALL$, then required that all $f$ edges of $H$ are present in $G_{n,p}$, and finally required that for each of the $s$ vertices in $SMALL$, their degree outside the copy of $H$ is at most $\frac{\log n}{100q}$.

\end{proof}

\begin{lem}\label{lem4}
W.h.p., for every $k \in \left[q-1,\frac{\log n}{100q}\right]$, there are less than $\n_k=\frac{e^{2\omega}(\log n)^{k-q+1}}{(k-1)!}$ vertices of degree $k$ in $G_{n,m}$.
\end{lem}
\begin{rem}
$\n_k$ is increasing in $k$ for this range, and for the largest $k=\frac{\log n}{100q}$ we have
$\n_k \lesssim %e^{\omega} \frac{(\log n)^k}{(k/e)^k \sqrt{2 \pi k}}
n^{\frac{\log (100eq)}{100q}}$.
\end{rem}
\begin{proof}
Fix $k$ and then we have
\begin{align*}
&\P(G_{n,p}\text{ has at least $\n_k$ vertices of degree at most }k)\\ 
&\leq \binom{n}{\n_k}\left(\sum_{\ell=0}^k\binom{n-\n_k}{\ell}p^\ell(1-p)^{n-\n_k-\ell}\right)^{\n_k}\\
&=\binom{n}{\n_k}\left((1+o(1))\binom{n-\n_k}{k}p^k(1-p)^{n-\n_k-k}\right)^{\n_k}\\
& \leq \left(\frac{ne}{\n_k} \times \frac{n^k}{k!} \left(\frac{\log n+(q-1)\log \log n -\omega}{n}\right)^k  e^{-\log n-(q-1)\log \log n  +\omega+o(1)} \right)^{\n_k}\\
&\leq\brac{\frac{e^{\omega +O(1)}}{(\log n)^{q-1}}\frac{(\log n+q\log\log n)^k}{k!\n_k}}^{\n_k}\\
&=\left(\frac{e^{-\omega +O(1)}}{k}\left(1+\frac{q\log \log n}{\log n}\right)^k \right)^{\n_k}\\
&\leq \left(e^{-\omega+O(1)} \frac{(\log n)^{kq/\log n}}{k}\right)^{\n_k}.
\end{align*}
The function $f(k)=\frac{(\log n)^{kq/\log n}}{k}$ is log-convex, and so $f$ is maximised at the extreme values of $k$ (specifically
$f(q-1)=e^{O(1)}>f\left(\frac{\log n}{100q}\right)=o(1)$). Hence,
$$\P(\exists k:G_{n,p}\text{ has at least $\n_k$ vertices of degree }k)\leq \sum_{k=q-1}^{\frac{\log n}{100q}}e^{-\omega\n_k/2}=o(1).$$
Applying \eqref{Prop2} we see that
$$\P(\exists k:G_{n,m}\text{ has at least $\n_k$ vertices of degree }k)=o(1),$$
which is stronger than required.
\end{proof}
\begin{lem}\label{lem3}
With probability $1-o(n^{-10})$,  $G_{n,m}$ has no vertices of degree $\geq 20\log n$.
\end{lem}
\begin{proof}
We will prove that w.h.p. $G_{n,p}$ has the stated property. We can then obtain the lemma by applying \eqref{Prop2}.
\begin{align*}
\P(\exists v:d(v) \geq 20 \log n)
& \leq n \binom{n-1}{20 \log n}p^{20 \log n}\\
& \leq n \left( \frac{en}{20 \log n} \frac{2 \log n}{n} \right)^{20 \log n}\\
& \leq n \left( \frac{e}{10} \right)^{20 \log n}\\
&= o(n^{-10}).
\end{align*}
\end{proof}
\section{Analysis of COL}
Let $\G=G_{m}$ and let $d(v)$ denote the degree of $v\in[n]$ in $\G$. Let
$$\th_v=\begin{cases}0&d(v)\geq q.\\1&d(v)=q-1.\end{cases}$$
\begin{lem}\label{prop1}
Suppose we run COL as described above.
Then w.h.p. $|C_v^{(m)}|=\th_v$ for all $v\in [n]$.
\end{lem}
In words, Lemma \ref{prop1} guarantees that the algorithm COL typically performs so that at time $m$, each vertex of degree at least $q$ has all colors present at its incident edges, while each vertex of degree $q-1$ has exactly one color missing. (It is well known that w.h.p. $\delta(G_m)=q-1$, see for example \cite{FK}, Section {\bf 4.2}.)

\begin{proof}
Fix $v$ and suppose $v$ has $k$ neighbours in LARGE, via edges $\{f_i = vu_i\}_{i=1}^k$. Then in general $d(v)-1\leq k \leq d(v)$ as small vertices do not share a path of length two. Also, when $v$ is small, $k=d(v)$.  Write $t(e)$ for the time $t \in \left[1,m\right]$ at which an edge $e$ appears in the random graph process, i.e. $t(e_i)=i$. Let $t_i=t(f_i)$ and assume that $t_i<t_{i+1}$ for $i>0$. We omit $i=1$ in the next consideration since $v$ will always get a color it needs by time $t_1$. (It may get a color before $t_1$ through an edge $vw$ where $w$ is not in LARGE.) Every time an $f_i,i\geq 2$, appears while $u_i$ needs no additional colors, $v$ gets a color it needs. So for $v$ to have $|C_v^{(m)}|>\th_v$ at the end of the process, this must happen at most $q-2-\th_v$ times, so there is certainly some set 
$$S=\set{i_1<i_2<\cdots<i_s}\subseteq [2,k]\text{ of }s=k-q+1+\th_v\text{ indices,}$$ 
whose corresponding edges $\set{f_i,i\in S}$ incident with $v$ satisfy $C_{u_i}^{(t_i)}\neq \emptyset$.  Let $\tS$ denote $\{t_i:i \in S\}$ and $\bU$ denote the sequence $u_1,u_2,\ldots,u_k$. In the following we will sum over $S$ and condition on the choices for $\tS$ and then estimate the probability that $C_{u_i}^{(t_i)} \neq \emptyset$ for $i\in S$. For a fixed $S$ there will be at least $\binom{m-k}{|S|+1}$ equally likely choices for the set $\set{t_i,i\in \set{1}\cup S}$. (We do not condition on $t_1$. The factor $t_{i_1}-1$ in \eqref{BB0} below will allow for the number of choices for $t_1$.) Let $\cL$ denote the occurrence of the bound of $20\log n$ on the degree of $v$ and its neighbors (see Lemma \ref{lem3}), and note that $\P(\cL)=1-o(n^{-10})$.

Taking a union bound over $S$, and letting
\begin{equation*}
A_i:=\bigg\{ C_{u_i}^{(t_i)} \neq \emptyset  \bigg\},
\end{equation*}
we have
\begin{align}
%p_v :=
\P(|C_v^{(m)}|>\th_v\mid\cL,\bU)
 &\leq
  \sum_{\substack{S \subset [2,k]\\ |S| = s} } \sum_{t_i:i \in \set{1}\cup S}
   \frac{1}{\binom{m-k}{k-q+2+\th_v}}\P\bigg(\bigwedge_{i \in  S}
    A_i \bigg\vert \tS,\bU,\cL\bigg)\\
 &\approx \sum_{\substack{S\subset [2,k] \\ |S| = s} } \sum_{t_i:i \in S}
  \frac{t_{i_1}-1}{\binom{m}{k-q+2+\th_v}}\P\bigg(\bigwedge_{i \in S}
   A_i \bigg\vert \tS,\bU,\cL \bigg),\label{BB0}
\end{align}
since there are $t_{i_1}-1$ choices for $t_1$ and $k^2=o(m)$, implying $\binom{m-k}{k-q+2+\th_v}\approx \binom{m}{k-q+2+\th_v}$, given $\cL$.
Next let
\begin{align*}
&Y_i=\{\text{edges of }u_i\text{ that appeared before }t_i\text{ excluding edges contained in $N^{(m)}(v)$}\},\\
&d_r=d(u_r)\text{ and }Z_{r}:=|Y_r|\text{ for }r=1,2,\ldots,s,\\
&\dS=\set{d_i:i\in S}.
\end{align*}
Now fix $\bU$ and $S$  and $\tS$ and $\dS$. 
\begin{rem}\label{rem1}
Going back to Algorithm COL, we observe that Step 5 implies that if $C_v^{(t)}\neq \emptyset$ then $uv$ is colored with a color in $C_v^{(t)}$ with probability at least $\frac{1}{q}$. This holds regardless of the previous history of the algorithm and also holds conditional on $\tS,\bU,\dS$. Indeed, the random bits used in Step 5 are independent of the history and are distinct from those used to generate the random graphs. The latter explains why we can condition on the future by fixing $\tS,\bU,\dS$. We condition on $\cL$ in order to control $s$ as $O(\log n)$.
\end{rem}
Then,
\begin{align}
&\P\left( A_{i_1} \wedge \dots \wedge A_{i_s}\mid\tS,\bU,\dS,\cL\right)\\
&=\sum_{z_s} \underbrace{\P(A_{i_s}\mid A_{i_1} , \dots, A_{i_{s-1}}, Z_s=z_s,\tS,\bU, \dS,\cL)}_{\leq\P(\Bin(z_s,q^{-1})\leq q-1) \text{ by Remark \ref{rem1}}}\P(A_{i_1},\dots, A_{i_{s-1}}, Z_s=z_s\mid\tS,\bU,\dS,\cL)\\
& \leq \sum_{z_s} g(z_s)\sum_{z_{s-1}}\P(A_{i_{s-1}}\mid A_{i_1},\dots,A_{i_{s-2}}, Z_{s-1}=z_{s-1},Z_s=z_s,\tS,\bU,\dS,\cL) \\
&\hspace{2in}\times\P(A_{i_1}, \dots, A_{i_{s-2}},Z_{s-1}=z_{s-1},Z_s=z_s\mid\tS,\bU,\dS,\cL) \\
&\leq \sum_{z_s,z_{s-1}}g(z_s)g(z_{s-1}) \P(A_{i_1}, \dots, A_{i_{s-2}},Z_{s-1}=z_{s-1},Z_s=z_s\mid\tS,\bU,\dS,\cL)\\
& \leq \sum_{z_s, \dots, z_1} g(z_s)\cdots g(z_2)\P(Z_r=z_r,r=1, \dots,s\mid\tS,\bU,\dS,\cL) \text{ (by induction)}\label{f1}
\end{align}
Here $g(z):=\P(\Bin(z,q^{-1})\leq q-1)$ for any $z\geq 0$.
\begin{claim}\label{cl1}
$$\P(Z_r=z_r, r=1,2, \dots, s\mid\tS,\bU,\dS,\cL)\leq \left(1+\tO(n^{-1})\right)\prod_{r=1}^s \frac{\binom{t_r}{z_r}\binom{m-t_r}{d_r-z_r}}{\binom{m}{d_r}},$$
where $\tO$ hides polylog factors.
\end{claim}
\proofstart
Fix $\frac{\log n}{100q} \leq d_1,d_2,\ldots,d_s=O(\log n)$ and $t_1,t_2,\ldots,t_s$. Then, for every $1\leq r\leq s$,
\begin{align}
\P(Z_r=z_r\mid Z_{r-1}=z_{r-1},\ldots, Z_1=z_1,\tS,\bU,\dS,\cL)
&\leq (1+o(n^{-10}))\frac{\binom{t_r}{z_r} \binom{m-t_r}{d_r-z_r}}{\binom{m-d_2-\cdots-d_{r-1}-s}{d_r}}\label{more}\\
&\leq \brac{1+\tO(n^{-1})}\frac{\binom{t_r}{z_r}\binom{m-t_r}{d_r-z_r}}{\binom{m}{d_r}}.\label{another}
\end{align}

{\bf Explanation for \eqref{more}:} The the first binomial coefficient in the numerator in \eqref{more} bounds the number of choices for the $z_r$ positions in the sequence where an edge contributing $Y_r$ occurs. This holds regardless of $z_1,z_2,\ldots,z_{r-1}$. The second binomial coefficient bounds the number of choices for the $d_r-z_r$ positions in the sequence where we choose an edge incident with $u_r$ after time $t_r$. Conversely, the denominator in \eqref{more} is a lower bound on the number of choices for the $d_r$ positions  where we choose an edge incident with $u_r$, given $d_1,d_2,\ldots,d_{r-1}$. We subtract the extra $s$ to (over)count for edges from $v$ to $u_{r+1},\ldots,u_s$. The factor $(1+o(n^{-10}))$ accounts for the conditioning on $\cL$.

Expanding $\P(Z_r=z_r,r=1, \dots, s\mid \tS,\bU,\dS,\cL)$ as a product of $s=O(\log n)$ of these terms completes the proof of Claim \ref{cl1}.
\proofend

Going back to \eqref{f1} we see that given $d_1,d_2,\ldots,d_s$,
\begin{align}
&\P\left( A_{i_1} \wedge \dots \wedge A_{i_s}\mid\tS,\bU,\dS,\cL\right)\\
&\lesssim \prod_{r=1}^s\sum_{z_r=0}^{d_r}\brac{\P(Bin(z_r,q^{-1})\leq q-1) \times\frac{\binom{t_r}{z_r}\binom{m-t_r}{d_r-z_r}}{\binom{m}{d_r}}}\\
&\leq\prod_{r=1}^s\sum_{z_r=0}^{d_r}\brac{C_1\binom{z_r}{\min\set{z_r,q-1}} \frac{1}{q^{q-1}}\brac{1-\frac{1}{q}}^{z_r}\times\frac{\binom{t_r}{z_r} \binom{m-t_r}{d_r-z_r}}{\binom{m}{d_r}}}\label{C1}\\
&\leq \prod_{r=1}^s\sum_{z_r=0}^{d_r}\brac{C_1\max\set{1,z_r^{q-1}}e^{-z_r/q} \times\frac{\binom{t_r}{z_r}\binom{m-t_r}{d_r-z_r}}{\binom{m}{d_r}}}.\label{AA1}
\end{align}

Here, $C_1=C_1(q)$ depends only on $q$. We will use constants $C_2, C_3, \dots$ in a similar fashion without further comment.

{\bf Justification for \eqref{C1}:} 
If $z_r\leq q-1$ then $\P(Bin(z_r,q^{-1})\leq q-1)=1$ and $C_1=eq^q$ will suffice.

If $q\leq z_r\leq 10q$ we use
$$\P(Bin(z_r,q^{-1})\leq q-1)\leq 1\text{ and }\binom{z_r}{q-1} \frac{1}{q^{q-1}} \brac{1-\frac{1}{q}}^{z_r}\geq \frac{1}{q^{q-1}} \brac{1-\frac{1}{q}}^{10q}$$

and $C_1=e^{20}q^q$ will suffice in this case.

If $z_r>10q$ then putting $a_i:=\P(Bin(z_r,q^{-1})=i)=\binom{z_r}{i}\frac1{q^i}\brac{1-\frac1q}^{z_r-i}$ for $i\leq q-1$ we see that
$$\frac{a_i}{a_{i-1}}=\frac{z_r-i+1}{i}\cdot\frac{1}{q-1}\geq \frac{z_r-q}{q^2}>\frac{z_r}{2q^2}\geq\frac{5}{q}.$$
So here
\begin{multline}
\P(\Bin(z_r,q^{-1})\leq q-1)=\sum_{i=0}^{q-1}a_i \leq a_{q-1} \left(1+\frac{2q^2}{z_r}+\dots+\left(\frac{2q^2}{z_r}\right)^{q-2}\right)\leq\\ \brac{1-\frac1q}^{1-q}\brac{\binom{z_r}{q-1} \frac{1}{q^{q-1}}\brac{1-\frac{1}{q}}^{z_r}} \frac{\bfrac{q}{5}^{q-1}-1}{\frac{q}{5}-1},
\end{multline}
and thus $C_1=(5q)^q$ suffices.

This completes the verification of  \eqref{C1}.

Now, writing $(t)_z$ for the \emph{falling factorial} $t!/(t-z)!=t(t-1)(t-2)\dots(t-z+1)$,
\begin{align}
\frac{\binom{t_r}{z_r}\binom{m-t_r}{d_r-z_r}}{\binom{m}{d_r}}
&= \binom{d_r}{z_r}\frac{(t_r)_{z_r}
(m-t_r)_{d_r-z_r}}{(m)_{d_r}}\\
&=\binom{d_r}{z_r}\prod_{i=0}^{z_r-1}\frac{t_r-i}{m-(d_r-z_r)-i}\cdot \prod_{i=0}^{d_r-z_r-1}\frac{m-t_r-i}{m-i}\\
&\leq \brac{1+O\bfrac{d_r^2}{m}}\binom{d_r}{z_r}\bfrac{t_r}{m}^{z_r}\brac{1-\frac{t_r}{m}}^{d_r-z_r}. \label{HH}
\end{align}
Observe next that if $z_r\geq q^2$ then
\beq{HHH}{
(z_r)_{q-1}=z_r^{q-1}\prod_{i=0}^{q-1}\brac{1-\frac{i}{z_r}}\geq z_r^{q-1}\brac{1-\frac{q^2}{2z_r}}\geq \frac{z_r^{q-1}}{2}.
}
It follows from \eqref{HH} and \eqref{HHH} that
\begin{align}
&\sum_{z_r=q^2}^{d_r}C_1z_r^{q-1}e^{-z_r/q}\times \frac{\binom{t_r}{z_r}\binom{m-t_r}{d_r-z_r}}{\binom{m}{d_r}}\\
&\leq 2C_1\sum_{z_r=q-1}^{d_r}(z_r)_{q-1}\binom{d_r}{z_r} \bfrac{t_re^{-1/q}}{m}^{z_r}\brac{1-\frac{t_r}{m}}^{d_r-z_r}\\
&\leq 2C_1(d_r)_{q-1}\bfrac{t_r}{m}^{q-1}\sum_{z_r=q-1}^{d_r} \binom{d_r-q+1}{z_r-q+1}\bfrac{t_re^{-1/q}}{m}^{z_r-q+1}\brac{1-\frac{t_r}{m}}^{d_r-z_r}\\
&\leq2C_1\bfrac{d_rt_r}{m}^{q-1}\brac{1-\frac{t_r}{m}\brac{1-e^{-1/q}}}^{d_r-q+1}\\
&\leq 2C_1\bfrac{d_rt_r}{m}^{q-1}\exp\set{-\frac{(d_r-q+1)t_r}{m}\brac{1-e^{-1/q}}}.
\end{align}
Furthermore, not forgetting
\begin{align}
\sum_{z_r=0}^{q^2-1}C_1\max\set{1,z_r^{q-1}}e^{-z_r/q}\times \frac{\binom{t_r}{z_r}\binom{m-t_r}{d_r-z_r}}{\binom{m}{d_r}} &\leq C_2\sum_{z_r=0}^{q^2-1} \frac{\binom{t_r}{z_r}\binom{m-t_r}{d_r-z_r}}{\binom{m}{d_r}} \\
&\leq C_3\sum_{z_r=0}^{q^2-1}t_r^{z_r}\cdot\frac{(m-t_r)^{d_r-z_r}}{(d_r-z_r)!}\cdot \frac{d_r!}{m^{d_r}}\\
&\leq C_3\sum_{z_r=0}^{q^2-1}\bfrac{d_rt_r}{m}^{z_r}e^{-(d_r-z_r)t_r/m}\\
&\leq C_4\psi\bfrac{d_rt_r}{m},\label{why}
\end{align}
where $\psi(x)=e^{-x}\sum_{z=0}^{q^2-1}x^z$. (Now $z_r\leq q^2$ and so the factor $e^{z_rt_r/m}\leq e^{q^2}$ can be absorbed into $C_4$.) Going back to \eqref{AA1} we have
\begin{multline}
\P\left( A_{i_1} \wedge \dots \wedge A_{i_s}\mid\tS,\bU,\dS,\cL\right)\leq\\ C_5^s\prod_{r=1}^s\brac{\bfrac{d_rt_r}{m}^{q-1} \exp\set{-\frac{d_rt_r}{m}\brac{1-e^{-1/q}}}+ \psi\bfrac{d_rt_r}{m}}.\label{BB1}
\end{multline}
It follows from \eqref{BB0} and \eqref{BB1} that, 
\begin{multline*}
p_{v}:=\P(|C_v^{(m)}|>\th_v\mid \tS,\bU,\dS,\cL)\\
\leq \sum_{\substack{S \subset [2,k] \\ |S|=s} } \sum_{t_i:i \in S} \frac{t_{i_1}C_5 ^s}{\binom{m}{s+1}} \prod_{r=1}^s\brac{\bfrac{d_rt_r}{m}^{q-1} \exp\set{-\frac{d_rt_r}{m}(1-e^{-1/q})}+\psi\bfrac{d_rt_r}{m}}.
\end{multline*}
Replacing a sum of products by a product of sums and dividing by $s!$ to account for repetitions, we get
\begin{align}
p_{v}&\leq \sum_{\substack{S \subset [2,k] \\ |S|=s} } \frac{C_5^s}{\binom{m}{s+1}s!} \prod_{r=2}^s\brac{\sum_{t=1}^m \brac{\bfrac{d_rt}{m}^{q-1}\exp\set{-\frac{d_rt}{m}(1-e^{-1/q})}+  \psi\bfrac{d_rt}{m}}}\\
& \hspace{2in}\times \brac{ \sum_{t=1}^m \brac{t\bfrac{d_1t}{m}^{q-1} \exp\set{-\frac{d_1t}{m}(1-e^{-1/q})}+t  \psi\bfrac{d_1t}{m}}}.
\end{align}
We now replace the sums by integrals. This is valid seeing as the summands have a bounded number of extrema, and we replace $C_5$ by $C_6$ to absorb any small error factors.
\begin{align}
p_{v}&\leq \sum_{\substack{S \subset [2,k] \\ |S|=s} } \frac{C_6^s}{\binom{m}{s+1}s!} \prod_{r=2}^s\brac{\int_{t=0}^\infty\left[ \bfrac{d_rt}{m}^{q-1}\exp\set{-\frac{d_rt}{m}(1-e^{-1/q})} +\psi\bfrac{d_rt}{m}\right]dt}\\
  & \hspace{2in} \times \brac{ \int_{t=0}^{\infty} \brac{t \bfrac{d_1t}{m}^{q-1}\exp\set{-\frac{d_1t}{m}(1-e^{-1/q})}+\psi\bfrac{d_1t}{m}}dt }\\
&=\sum_{\substack{S \subset [2,k] \\ |S|=s} } \frac{C_6^s}{\binom{m}{s+1}s!} \prod_{r=2}^s\brac{\frac{m}{d_r}\int_{x=0}^\infty (x^{q-1}e^{-(1-e^{-1/q})x}+\psi(x))dx}\\
  & \hspace{3in} \times \frac{m^2}{d_1^2}\brac{ \int_{x=0}^{\infty} \brac{x^q\exp\set{-(1-e^{-1/q})x}+ x\psi(x)}dx } \\
&\leq\sum_{\substack{S \subset [2,k] \\ |S|=s} } \frac{C_6^s}{\binom{m}{s+1}s!}\cdot\bfrac{ C_7m}{\min_r\{d_r\}}^{s+1} \\
&\leq \frac{C_8^k}{(\log n)^{k-q+2+\th_v}}.
\end{align}
Applying Lemmas \ref{lem4} and \ref{lem3} and removing the conditioning on $\tS,\bU,\dS,\cL$ we see that with $k_0=\frac{\log n}{100q}$,
\begin{align*}
&\P(\exists v:|C_v^{(m)}|> \th_v)\\
&\leq \P(\neg\cL)+\sum_{k=q-1}^{k_0}\frac{e^{2\omega}(\log n)^{k-q+1}}{(k-1)!}\times \frac{ C_8^k}{(\log n)^{k-q+2+\th_v}}+n\sum_{k=k_0}^{20\log n} \frac{C_8^k}{(\log n)^{k-q+2}}\\
&\leq o(1)+\frac{e^{2\omega}}{\log n}\sum_{k\geq q-1} \frac{C_8^k}{(k-1)!}+n\sum_{k=k_0}^{20\log n}\frac{C_8^k}{(\log n)^{k/2}}\\
&\leq o(1)+\frac{C_{9}e^{2\omega+C_9}}{\log n}\\
&=o(1).
\end{align*}
\end{proof}
We show next that at time $m$, w.h.p. sets of size up to $\Omega(n)$ have large neighbourhoods in every color.\\
We first prove that typically ``large-degree vertices have large degree in every color'': let $d^*_c(v)$ denote the number of edges incident with $v$ that COL colors $c$, except for those edges that are colored in Step 3.
\begin{thm}\label{th2}
There exists $\epsilon=\epsilon(q) >0$ such that w.h.p. on completion of COL every $v \in LARGE$
has $d_c^*(v) \geq \frac{\epsilon \log n}{1000q}$ for all $c \in [q]$.
\end{thm}
Suppose we define a vertex to be $small_c$ if it has $d_c(v) \leq \frac{\epsilon \log n}{1000q}$.
Theorem \ref{th2} says w.h.p. the set of $small_c$ vertices $SMALL_c \subset SMALL$ so that by Lemma \ref{lem1}, $G$ does not contain any $small_c$ structures of constant size. Here a $small_c$ structure is a small structure made up of $small_c$ vertices.

The proof of Theorem \ref{th2} will follow from Lemmas \ref{lemx}, \ref{lemy} and \ref{lemz} below.
\begin{lem}\label{lemx}
There exists $\delta = \delta(q)>0$ such that the following holds w.h.p.: Let $Full',Full$ be as in Definition \ref{Fudef}. Then $|Full'| \geq n-\frac{203 qn}{\epsilon \log n}$, and $|Full| \geq n-n^{1-\delta}$.
\end{lem}
\begin{proof}
We first note that for $v\in [n]$, that if $t_\e=\epsilon n \log n$ then
\beq{deg1}{
\P\brac{d^{(t_\e/2)}(v)< \l_0:=\frac{\epsilon \log n}{100}}\leq 3n^{-\e/2} < n^{-\e/3}.
}
Indeed, with $p_1 =\frac{t_{\e}/2}{\binom{n}{2}}$ we see that, in the random graph model $G_{n,p_1}$:
\beq{deg2}{
\P\brac{d(v)<\l_0}=\sum_{i=0}^{\l_0-1}\binom{n}{i}p_1^i(1-p_1)^{n-i}\leq 2\binom{n}{\l_0}p_1^{\l_0}(1-p_1)^{n-\l_0}\leq 2\bfrac{nep_1}{\l_0}^{\l_0}n^{-\e+o(1)}\leq n^{-\e/2}.
}
The first inequality follows from the fact that the ratio of succesive summands in the sum is at least $(n-\l_0)p_1/\l_0>50$.

Equation \eqref{deg1} now follows from \eqref{Prop2} (with $p$ replaced by $p_1$) and \eqref{deg2}.

Thus the Markov inequality shows that with probability at least $1-n^{-\e/3}$, at least $n-n^{1-\e/6}$ of the vertices $v$ have $d^{(t_\e/2)}(v)\geq \frac{\epsilon \log n}{100}$. Now note that at most $qn$ of the first $t_\e/2$ edges were restricted in color by being incident to at least one needy vertex. This is because each time a needy vertex gets an edge incident to it, the total number of needed colors decreases by at least one. Therefore at most $\frac{200qn}{\epsilon  \log n}$ of these vertices $v$ have fewer than $\frac{\epsilon \log n}{200}$ of their $\geq \frac{\epsilon \log n}{100}$ initial edges colored completely at random, as in Step 4 of COL. Hence, there are at least $n-\frac{201qn}{\epsilon\log n}$ vertices $v$ which have $d^{(t_\e/2)}(v)\geq \frac{\epsilon \log n}{100}$ and $\frac{\e \log n}{200}$ edges of fully random color. For such a $v$, and any color $c$,
\beq{zz2}{
\P\left(d_c^{(t_\e/2)}(v)<\frac{\epsilon \log n}{1000q}\right)\leq \P \left(Bin\left(\frac{\epsilon \log n}{200}, \frac{1}{q} \right) \leq \frac{\epsilon \log n}{1000q} \right) \leq \exp\set{-\frac12\cdot \frac{16}{25}\cdot\frac{\epsilon \log n}{200q}}\leq n^{-\e/1000q}.
}
So $\P(v \notin Full') \leq qn^{-\e/1000q}$, and the Markov inequality shows that  w.h.p.
\beq{zz1}{
|Full'| \geq n-\frac{201qn}{\epsilon\log n}-n^{1-\e/2000q}\geq n-\frac{202 q n }{\epsilon \log n}.
}
Now, for $v \notin Full'$, let $S(v):=Full' \setminus N^{(t_\e/2)}(v)$. Since $d^{(t_\e/2)}(v)\leq d(v)
\leq 20 \log n$, we have $|S(v)| \geq n-\frac{203 q n }{\epsilon \log n}$. Furthermore,  every $w \in S(v)$ is no longer needy, and so among the next $t_\e/2$ edges, at most $q$ of the edges between $v$ and $S(v)$ have their choice of color restricted by $v$, and the rest are colored
randomly as in Step 4 of COL. Now $\P(|e(v,S(v))|<\frac{\epsilon \log n}{100})=O(n^{-\e/3})$ by a similar calculation to \eqref{deg2}. Conditioning on $|e(v,S(v))|\geq \frac{\epsilon \log n}{100}$
we have $\P(v \notin Full) \leq qn^{-\e/1000q}$ by a similar calculation to \eqref{zz2} and so
$|Full| \geq n-n^{1-\e/2000q}$ with probability $\geq 1-O(n^{-\e/2000q})$ by the Markov inequality.
\end{proof}

We are working towards showing that vertices with low degree in some color must have have a low overall degree.
The point is that all $Full$ vertices no longer need additional colors later than $t_{\epsilon}=\epsilon n \log n$, so any new edge connecting $Full$ to $V \setminus Full$ after time $t_\e$ has its color determined by the vertex not in $Full$, as in Step 2 of COL.
Indeed, suppose a vertex $v \notin Full$ has at least $\frac{ \epsilon \log n}{400}$ edges to $Full$ after time $t_\e$. Then $v$ gets at least $\frac{\epsilon \log n}{400q}>\frac{\epsilon \log n}{1000q}$ edges of every color incident with it.
\begin{lem}\label{lemy}
W.h.p. there are no vertices $v \notin Full$ with at least $\frac{\epsilon \log n}{200}$ edges after time $t_\epsilon$ i.e., $d^{(m)}(v)-d^{(t_\e)}(v)\geq\frac{\epsilon \log n}{200}$ but with at most $\frac{\epsilon \log n}{400}$ of these edges to $Full$.
\end{lem}
\begin{proof}
Take any vertex $v \notin Full$ and consider the first $\frac{\epsilon \log n}{200}$ edges incident to $v$ after time $t_\epsilon$. We must estimate the probability that at least half of these edges are to vertices not in $Full$. We bound this by
$$\binom{\e\log n/200}{\e\log n/400}\bfrac{n^{1-\delta}}{n-20\log n}^{\e\log n/400}=o(1/n).$$
We subtract off a bound of $20\log n$ on the number of edges from $v$ to $Full$ in $E_{t_\e}$. Note that we do not need to multiply by the number of choices for $Full$, as $Full$ is defined by the first $t_\e$ edges. There at most $n$ choices for $v$ and so the lemma follows.
\end{proof}
\begin{lem}\label{lemz}
There are no \emph{large} vertices $v$ with $d^{(m)}(v)-d^{(t_\epsilon)}(v)<\frac{\epsilon \log n}{200}$.
\end{lem}
\begin{proof}
Any $v$ satisfying these conditions must have $d^{(t_\epsilon)}(v) \geq \frac{\log n }{200q}$, if $\epsilon \leq 1/q$ say.
However, with $p_2=\frac{t_\e}{\binom{n}{2}}$ we have that in the random graph model $G_{n,p_2}$,
\beq{qdef}{
\P\left(d(v) \geq\frac{\log n}{200q} \right)\leq \binom{n}{\log n/200q}p_2^{\log n/200q} \leq (400qe\e)^{\log n/200q}=o(n^{-2}),
}
for $\e$ sufficiently small.

The result follows by taking a union bound over choices of $v$ and using \eqref{Prop2} (again noting $\binom{n}{2}p_2 = t_\epsilon \rightarrow \infty$).
\end{proof}
{\bf Proof of Theorem \ref{th2}:}  It follows from Lemmas \ref{lemy}, \ref{lemz} that every large vertex has at least $\frac{\epsilon \log n}{400}$ edges to $Full$ that occur after time $t_\epsilon$. These edges will provide all needed edges of all colors.
\qed
\medskip

It is known that w.h.p. $m\leq \t_{q}\leq m'=m+2\omega n$, see Erd\H{o}s and R\'enyi \cite{ER61}. We have shown that at time $m$ all vertices, other than vertices of degree $q-1$, are incident with edges of all colors. Furthermore, vertices of degree $q-1$ are only missing one color. As we add the at most $2\omega n$ edges needed to reach $\t_{q}$ we find (see Claim \ref{clq} below) that w.h.p. the edges we add incident to a vertex $v$ of degree $q-1$ have their other end in LARGE. As such COL will now give vertex $v$ its missing color.
\begin{claim}\label{clq}
W.h.p. an edge of $E_{m'}\setminus E_m$ that meets a vertex of degree $q-1$ in $G_m$ has its other end in LARGE.
\end{claim}
\begin{proof}
It follows from Lemma \ref{lem4} that at time $m$ and later there are w.h.p. at most $e^{2\omega}$ vertices of degree $q-1$ and at most 
$$M=\sum_{k=q-1}^{\log n/100q}\frac{e^{2\omega}(\log n)^{k-q+1}}{(k-1)!}\leq 2e^{2\omega}\bfrac{e\log n}{\log n/100q}^{-q+1+\log n/100q}\leq n^{1/3}$$ 
vertices in SMALL. (The first inequality follows from the fact that the summands grow by a factor of at least $100q$.) 

Thus the probability that there is an edge contradicting the claim is at most
$$2\omega n\times \frac{e^{2\omega}\times n^{1/3}}{\binom{n}{2}-m'}=o(1).$$
\end{proof}
We remind the reader that $q=2\s$ where we only use $\s$ colors. We apply the above analysis by identifying colors $mod \; \s$. We therefore have the following:  
\begin{cor}\label{cord}
W.h.p. the algorithm COL applied to $G_{\t_{2\s}}$ yields a coloring for which $d_c^*(v)\geq 2$ for all $v\in [n]$.
\end{cor}
\begin{proof}
We can see from the above that w.h.p. at time $\t_{2\s}$ we have that $d_c(v)\geq 2$ for all $c\in [\s],v\in [n]$. Furthermore, by construction, for each $c\in [q], v\in [n]$ the first edge incident with $v$ that gets color $c$ will be in $E_c^*$. (The only time we place an edge in $E_c^+$ is when it joins two full vertices.) 
\end{proof}
{From} now on we think in terms of $\s$ colors.

\subsection{Expansion}
%Let $\G_c^*$ denote the subgraph of $G_m$ induced by the edges of color $c$ that are not in $E^+$ (but could be in $E^*$). 
For a set $S\subseteq [n]$ we let
$$N_c^*(S)=\set{v\notin S:\exists u\in S\ s.t.\ uv\in E_c^*} \subset N_c(S).$$
Let 
$$\alpha = \frac{1}{10^6 q}.$$
\begin{lem}\label{th3}
Then w.h.p. $|N_c^*(S)| \geq 19|S|$ for all $S \subset LARGE$, $|S| \leq \alpha n$.
\end{lem}
\begin{claim}\label{clprev}
At time $m$, for every $R \subset V(G)$ with $|R| \leq \frac{n}{(\log n)^3}$, there are w.h.p. at most $2|R|$ edges within $R$ of color $c$ for every $c \in [q]$.
\end{claim}
{\em Proof of Claim:}
We will show that w.h.p. every such $R$ does not have this many edges irrespective of color. Note that the desired property is monotone decreasing, so it suffices to use \eqref{Prop2} and show this occurs w.h.p. in $G_{n,p}$:
\begin{align*}
& \P\left( \exists |R|\leq \frac{n}{(\log n)^3}: |E(G[R])|>2|R| \right)\\
& \leq \sum_{r=4}^{n/(\log n)^3} \binom{n}{r} \binom {\binom{r}{2}}{2r} p^{2r} \\
& \leq \sum_{r=4}^{n/(\log n)^3} \left( \frac{ne}{r} \left(\frac{re^{1+o(1)} \log n}{4n} \right)^2 \right)^r \\
& \leq \sum_{r=4}^{n/(\log n)^3}\left( \frac{r}{n} \cdot \frac{e^{3+o(1)}(\log n)^2}{16}\right)^r = o (n^{-3}).
\end{align*}
\qed

{\em Proof of Lemma \ref{th3}:}\\
\textbf{Case 1:} $|S| \leq \frac{n}{(\log n)^4}$.\\
We may assume that $S \cup N_c^*(S)$ is small enough for Claim \ref{clprev} to apply (otherwise $|N_c^*(S)| \geq \frac{n}{(\log n)^3}-\frac{n}{(\log n)^4}$ so that $S$ actually has logarithmic expansion in color $c$). Then, using $e_c$ to denote the number of edges in color $c$, and using Theorem \ref{th2},
 \begin{equation*}
\frac{ \epsilon \log n}{1000q}|S|\leq\sum_{v \in S}d_c^*(v)=2e_c(S)+e_c (S,N_c^*(S))\leq  4|S|+2|N_c^*(S) \cup S|.
\end{equation*}
Hence,
$$|N_c^*(S)| \geq \frac{\epsilon \log n}{2001q}|S|\geq 19|S|,$$
which verifies the truth of the lemma for this case.

\medskip
\textbf{Case 2:}
$\frac{n}{(\log n)^4} \leq |S| \leq \frac{n}{50 \log n}$.\\
Let
$$m_+:=\frac{n \log n}{8q}.$$
Let $E_c^+,E_c^*$ denote the edges of $E^+,E^*$ respectively, which are colored $c$.
We begin by proving
\begin{claim}\label{clxx}
$|E_c^+|,|E_c^*|\geq m_+$ w.h.p.
\end{claim}
\begin{proof}
Once $Full$ has been formed, it follows from Lemma \ref{lemx}, that at most $(n^{1-\delta}(n-n^{1-\delta}))+\binom{n^{1-\delta}}{2} < 2n^{-\delta} \binom{n }{2}$ spaces remain in $E(Full, V \setminus Full)$ or $E(V\setminus Full)$.
For each of the $m-t_{\epsilon} \sim (\frac{1}{2}-\epsilon)n \log n$ edges appearing thereafter,
since $\lesssim n \log n <n^{-\delta} \binom{n}{2}$  edges have been placed already,
each has a probability $\geq 1-4n^{-\delta}$ of having both ends in $Full$, independently of what has happened previously.
Applying the Chernoff bounds (see for example \cite{FK}, Chapter 21.4)
we see that the probability that fewer than $\frac{1}{3} n \log n$ of these $(\frac{1}{2}-\epsilon)n \log n$ edges were between vertices in $Full$ is
at most $e^{-\Omega(n\log n)}$. We remind the reader that every edge with both endpoints in $Full$ is randomly colored and placed in $E^+$ or $E^*$ in Step 3 of COL.

So, we may assume there are at least $\frac{1}{3q} n \log n$ of these edges in $E^+ \cup E^*$ of color $c$ in expectation and then the Chernoff bounds imply that there are at least $\frac{1}{8q} n \log n=m^+$ w.h.p. in both $E^+$ and $E^*$.
\end{proof}
Suppose there exists $S$ as above with $|N_c^*(S)|< \frac{\log n}{ 1000 q} |S|$. For $F:=S \cap Full$,
note that $|F| \geq |S| - n^{1-\delta}=|S|(1-o(1))$. Therefore $|N_c^*(F) \cap Full| < \frac{ \log n}{1000 q} |S| \leq \frac{\log n}{999q} |F|$. We will show that w.h.p. there are no such $F\subseteq Full$.

We consider the graphs $H_1=G_{|Full|,m_+}\setminus E_{t_\e}$ and the corresponding independent model $H_2=G_{|Full|,p_+}\setminus E_{t_\e}$ where $p_+\sim \frac{\log n}{4qn}$. We will show that w.h.p. $H_2$ contains no set $F$ of the postulated size and small neighborhood. Together with \eqref{Prop2} (and $\binom{|Full|}{2}p_+ \rightarrow \infty$) this implies that w.h.p. $H_1$ has no such set either. Note that by Lemma \ref{lem3}, we see that w.h.p. at most $20|F|\log n$ edges of $E_{t_\e}$ are incident with $F$. This calculation is relevant because $(E^*\setminus E_{t^*})$'s only dependence on $E_{t_\e}$ is that it is disjoint from it.

Hence, in $H_2$,
\begin{align*}
\P(\exists F)& \leq \sum_{f =(n-o(n)) / (\log n)^4}^{n / 50 \log n} \sum_{k=0}^{\frac{\log n}{999q}f} \binom{|Full|}{f} \binom{|Full|}{k} f^k p_+^k(1-p_+)^{(|Full|-k)f-20f\log n} \\
& \leq \sum_{f=(n-o(n)) / (\log n)^4}^{n / 50\log n} \sum_{k=0}^{\frac{\log n}{ 999q}f}\underbrace{\left(\frac{ne}{f} \right)^f\left(\frac{nf}{k}\cdot\frac{ \log n}{qn}     \right)^ke^{-nfp_+ (1-o(1))}}_{u_{f,k}}.
\end{align*}
Here, the ratio
\begin{equation*}
\frac{u_{f,k+1}}{u_{f,k}}=\frac{f \log n}{q(k+1)}\left( \frac{k}{k+1}\right)^{k}\geq 999/e.
\end{equation*}
Therefore,
$$\P(\exists F) \leq2 \sum_{f \sim n / (\log n)^4}^{n / 50 \log n}\left(\frac{ne}{f}\cdot  ( 999)^{\frac{\log n}{999 q}}n^{-1/5q}\right)^f\leq 2n \left(3(\log n)^4n^{-1/10q}\right)^{\frac{(1-o(1))n}{(\log n)^4}}=o(1).$$
\textbf{Case 3:}
$\frac{n}{ 50 \log n} \leq |S| \leq \frac{n}{10^6q}$.\\
Choose any $S_1 \subset S$ of size $\frac{n}{50 \log n}$, then
$$|N_c^*(S)| \geq |N_c^*(S_1)|-|S|\geq \frac{\log n}{1000q}\cdot \frac{n}{50 \log n}-\frac{n}{10^6 q}= 19 \alpha n\geq  19|S|.$$
\qed

The following corollary applies to the subgraph of $G_{\t_{2\s}}$ induced by $E^*_c$.
\begin{cor}\label{corrr}
W.h.p. $|N_c^*(S)|\geq 2|S|$ for all $S \subset V(G)$ with $|S| \leq \alpha n$.
\end{cor}
\begin{proof}
We know from Corollary \ref{cord} that w.h.p. at time $m$ every vertex $v$ has $d_c^*(v) \geq 2$. Let $S_2=S \cap LARGE$, $S_1 = S \setminus S_2$. Then
\begin{align*}
|N_c^*(S)| &= |N_c^*(S_1)|+|N_c^*(S_2)| - |N_c^*(S_1)\cap S_2|- |N_c^*(S_2)\cap S_1| - |N_c^*(S_1) \cap N_c^*(S_2)| \\
& \geq |N_c^*(S_1)|+|N_c^*(S_2)|-|S_2|- |N_c^*(S_2)\cap S_1| - |N_c^*(S_1) \cap N_c^*(S_2)|.
\end{align*}
Clearly, $|S_2| \leq |S|\leq \alpha n$, and so Lemma \ref{th3} gives $|N_c^*(S_2)| \geq 19 |S_2|$, w.h.p. Also, recall from Lemma \ref{lem1} that w.h.p. there are no small structures in $G_m$ and since $SMALL_c \subset SMALL$ w.h.p., this means there aren't any small-$c$-structures either. In particular,
\begin{itemize}
\item No $small_c$ vertices are adjacent and there is no path of length two between $small_c$ vertices which implies that $|N_c^*(S_1)| \geq 2 |S_1|$ and $|N_c^*(S_2) \cap S_1| \leq |S_2|$.
\item In addition, there is no $C_4$ containing a $small_c$ vertex, and no path of length 4 between $small_c$ vertices. This means that $|N_c^*(S_1) \cap N_c^*(S_2)| \leq |S_2|$.
\end{itemize}
We deduce that
$|N_c^*(S)| \geq 2|S_1|+19|S_2|-3|S_2| \geq 2|S|$.
\end{proof}

Recall $\G_c^*$ is the subgraph induced by edges of color $c$ that are not in $E^+$.
\begin{cor}\label{COR}
W.h.p. $\G_c^*$ is connected for every $c \in [q]$.
\end{cor}
\begin{proof}
If $[S,V \setminus S]$ is a cut in $\G_c^*$ then Corollary \ref{corrr} implies that $|S|,|V\setminus S| \geq \alpha n$. Let $F = S \cap Full$. Since $|V\ \setminus Full| \leq n^{1-\delta}< \frac{\alpha n}{2} $ we see that $|F|,|Full \setminus F|\geq \frac{\alpha n}{2}$. As in \textbf{Case 2} of Lemma \ref{th3} we show that w.h.p. no such $F$ exists by doing the relevant computation in
$H_2$ with $p_+\sim \frac{\log n}{4qn}$:
\begin{align}
\P \left(\exists F\right)& \leq 2 \sum_{f=\frac{\alpha n}{2}}^{|Full|-\frac{\alpha n}{2}} \binom{|Full|}{f}
(1-p_+)^{f(|Full|-f)-t_\epsilon}\label{sub}\\
& \leq n 2^{n} (1-p_+)^{\frac{\alpha n}{2}  (n-n^{1-\delta}-\frac{\alpha}{2}n)-o(n^2)}\\
& \leq n2^n e^{-\alpha n\log n/10q} =o(1).
\end{align}
We subtract $t_\e$ from $f(|Full|-f)$ because we do not include the first $t_\e$ edges in this calculation. This is because $Full$ depends on them.
\end{proof}
\section{Rotations}
We now use $E_c^+$ to build the Hamiltonian cycles for every color $c$ using P\'{o}sa rotations. We let $G_c$ denote the graph induced by the edges of color $c$. Given a path $P=(x_1,x_2,\ldots,x_k)$ and an edge $x_ix_k, 2\leq i\leq k-2$ we say that the path $P'=(x_1,\ldots,x_i,x_k,\ldots,x_{i+1})$ is obtained from $P$ by a rotation with $x_1$ as the fixed endpoint.

For a path $P$ in $G_c$ with endpoint $a$ denote by $END(a)$, the set of all endpoints of paths obtainable from $P$ by a sequence of P\'{o}sa rotations with $a$ as the fixed endpoint.
In this context, P\'osa \cite{Po} shows that $|N_c(END(a))|<2|END(a)|$. This is assuming that in the course of executing the rotations, no simple extension of our path is found. It follows from Corollary \ref{corrr} that w.h.p. $|END(a)|\geq \alpha n$. For each $b\in END(a)$ there will be a path $P_b$ of the same length as $|P|$ with endpoints $a,b$. We let $END(b)$ denote the set of all endpoints of paths obtainable from $P_b$ by a sequence of P\'{o}sa rotations with $b$ as the fixed endpoint. It also follows from Corollary \ref{corrr} that w.h.p. $|END(b)|\geq \alpha n$ for all $b\in END(a)$. Let $END(P)=\set{a}\cup END(a)$.

An edge $u=\set{x,y}$ of color $c$ with $y\in END(x)$ is called a {\em booster}. Let $P_{x,y}$ be the path of length $|P|$ from $x$ to $y$ implied by $y\in END(x)$. Adding the edge $u$ to $P_{x,y}$ will either create a Hamilton cycle or imply the existence of a path of length $|P|+1$ in $G_c$, after using Corollary \ref{COR}. Indeed, if the cycle $C$ created is not a Hamilton cycle, then the connectivity of $\G_c^*$ implies that there is an edge $u=xy$ of color $c$ with $x\in V(c)$ and $y\notin V(C)$. Then adding $u$ and removing an edge of $C$ incident to $x$ creates a path of length $|P|+1$.

We start with a longest path in $\G_c^*$ and let $E_c^+=\set{f_1,f_2,\ldots,f_\ell}$ where w.h.p. $\ell\geq m_+=\frac{n\log n}{8q}$, see Claim \ref{clxx}. A round consists of an attempt to find a longer path than the current one or to close a Hamilton path to a cycle. Suppose we start a round with a path $P$ of length $k$. We use rotations and construct many paths. If one of these paths has an endpoint with a neighbor outside the path then we add this neighbor to the current path and start a new round with a path of length $k+1$. Here we only use edges not in $E_c^+$. Failing this we compute $END(P)$ and look for a booster in $E_c^+$. In the search for boosters we start from $f_r$ assuming that we have already examined $f_1,f_2,\ldots,f_{r-1}$ in previous rounds. Now $f_r$ is chosen uniformly from $(1-o(1))\binom{n}{2}$ pairs and so the probability it is a booster is at least $\beta=(1-o(1))\alpha^2$. It is clear that at most $n$ boosters are needed to create a Hamilton cycle. Adding a booster increases the length of the current path by one, or creates a Hamilton cycle. So the probability we fail to find a Hamilton cycle of color $c$ is at most $\P(Bin(m_+,\beta)\leq n)=o(1)$. We can inflate this $o(1)$ by $\sigma$ to show that w.h.p. we find a Hamilton cycle in each color, completing the proof of Theorem \ref{th1}.

\section{Concluding remarks}
In this paper we studied a very natural variant of the classical
problem of the appearance of $\s$ edge disjoint Hamilton cycles in a random graph process. 
We showed that one can color the edges of the process online so that every color class has a Hamilton cycle exactly at the moment when
the underlying graph has $\s$ edge disjoint ones. 

The paper \cite{BF} shows that at the hitting time $\t_{2\s+1}$ there will w.h.p. be $\s$ edge disjoint Hamilton cycles plus an edge disjoint matching of size $\lfloor n/2\rfloor$. It is straightforward to extend this result to the online situation. It should be clear that at time $\t_{2\s+1}$ COL can be used to construct w.h.p. $E_c^*,E_c^+,c=1,2,\s+1$ such that $E_c^*\cup E_c^+$ induce Hamiltonian graphs for $1\leq c\leq \s$ and $d_{\s+1}^*(v)\geq 1$ for $v\in[n]$. For color $\s+1$, we replace the statement of Corollary \ref{corrr} by 
\beq{last1}{
\text{W.h.p. $|N_{\s+1}^*(S)|\geq |S|$ for all $S \subset V(G)$ with $|S| \leq \alpha n$. }
}
We then replace rotations by alternating paths, using $E_{\s+1}^+$ as boosters. The details are as described in Chapter 6 of \cite{FK}. In outline, let $G=(V,E)$ be a graph without a matching of size $\lfloor|V(G)|/2\rfloor$. For $v\in V$ such that $v$ is isolated by some maximum matching, let
$$A(v)=\set{w\in V:w\neq v\text{ and }\exists\text{ a maximum matching of $G$ that isolates $v$ and $w$}}.$$
We use the following lemma 
\begin{lem}\label{matchlemma}
Let $G$ be a graph without a matching of size $\lfloor|V(G)|/2\rfloor$. Let $M$ be a maximum matching of $G$. If $v\in V$ and $A(v)\neq\emptyset$ then $|N_G(A(v))|<|A(v)|$.
\end{lem}
We start with a maximum matching $M$ of $\G_{\s+1}^*$. Suppose that $v$ is not covered by $M$. Using \eqref{last1}, we see that w.h.p. $|A(v)|\geq \alpha n$. Further, if $u\in A(v)$ and $uv\in E_{\s+1}^+$ then adding this edge gives a larger matching. Also, because $u$ is isolated by a maximum matching, there is a corresponding set $A_u$ of size at least $\alpha n$ such if $w\in A_u$ and $uw\in E_{\s+1}^+$ then we can find a larger matching. Therefore we have $\Omega(n^2)$ boosters and the proof is similar to that for Hamilton cycles.

There are several related problems which can likely be treated using our approach. One potential application for our technique is to show that for any fixed positive integer $k$ and any decomposition $k=k_1+...+k_s$ into the sum of $s$ positive integers, there is an online algorithm, coloring the edges of a random graph process in $s$ colors so that exactly at the hitting time $\tau_k$ the $i$-th color forms a $k_i$-connected spanning graph for $i=1,\ldots,s$. In general, one can generate many more interesting problems by considering the online Ramsey version of other results in the theory of random graphs.

\end{document}